\documentclass{article}

\usepackage{amsfonts,amssymb}
\usepackage{amsmath, amsthm}
\usepackage{graphicx}
\usepackage{cite}
\usepackage{enumerate}

\newtheorem{theorem}{Theorem}[section]

\newtheorem{definition}[theorem]{Definition}

\newtheorem{cons}{Consequence}[section]
\newtheorem{ass}{Assumptions}[section]

\begin{document}
\title{Existence and uniqueness of mild solutions for BGK models for gas mixtures of polyatomic molecules}
\author{Marlies Pirner}
%
\date{}
\maketitle
\begin{abstract} We consider two models for a two component gas mixture with translational and internal energy degrees of freedom described by a BGK approximation assuming that the number of particles of each species remains constant. The two species are allowed to have different degrees of freedom in internal energy and are modelled by a system of kinetic BGK equations featuring two interaction terms to account for momentum and energy transfer between the species. We consider the two models in the literature \cite{Pirner5} and \cite{Pirner9} and prove that the models satisfy a different equation of state than the ideal gas law. Moreover, we prove that for these models there exist a unique positive solution.  \end{abstract}
%
%

\\

\textbf{keywords:} multi-fluid mixture, kinetic model, BGK approximation, polyatomic molecules, existence, uniqueness, positivity
%

\section{Introduction}
 
 In this paper we shall concern ourselves with a kinetic description of gas mixtures for polyatomic molecules. In the case of mono atomic molecules and two species this is traditionally done via the Boltzmann equation for the density distributions $f_1$ and $f_2$, see for example \cite{Cercignani, Cercignani_1975}. Under certain assumptions the complicated interaction terms of the Boltzmann equation can be simplified by a so called BGK approximation, consisting of a collision frequency multiplied by the deviation of the distributions from local Maxwellians. This approximation should be constructed in a way such that it  has the same main properties of the Boltzmann equation namely conservation of mass, momentum and energy, further it should have an H-theorem with its entropy inequality and the equilibrium must still be Maxwellian.  BGK  models give rise to efficient numerical computations, which are asymptotic preserving, that is they remain efficient even approaching the hydrodynamic regime \cite{Puppo_2007, Jin_2010,Dimarco_2014, Bennoune_2008,  Bernard_2015, Crestetto_2012}. Evolution of a polyatomic gas is very important in applications, for instance air consists of a gas mixture of polyatomic molecules. But, most kinetic models modelling air deal with the case of a mono atomic  gas consisting of only one species. 
 
  In the literature one can find two types of single species models for polyatomic molecules. There are models which contain a sum of collision terms on the right-hand side corresponding to the elastic and inelastic collisions. Examples are the models of Rykov \cite{Rykov}, Holway \cite{Holway} and Morse \cite{Morse}. The other type of models contain only one collision term on the right-hand side taking into account both elastic and inelastic interactions. Examples for this are Bernard, Iollo, Puppo \cite{Bernard}, the model of Andries, Le Tallec, Perlat, Perthame \cite{Perthame} or the model by Bisi and Caceres \cite{Bisi} modelling chemical interactions. Furthermore, for gas mixtures there are  models \cite{Pirner5} and \cite{Pirner9}  extending \cite{Bernard}, where a gas mixture of polyatomic molecules is considered. The models in \cite{Pirner5} and \cite{Pirner9} allow the two species to have different degrees of freedom in internal energy. 
  
 In this paper, we consider  the models described in \cite{Pirner5} and \cite{Pirner9} and prove that we are able to derive a more generalized equation of state. This is an important issue for example when you want to describe atmospheric re-entry problems, see \cite{Luc}.

For the models in \cite{Pirner5} and \cite{Pirner9} conservation properties and an H-Theorem are proven. For a discussion of the physcial relevance of the two models, see \cite{Pirner9}. In a polyatomic gas, one has two different types of relaxation processes, one has relaxation of the distribution function to a Maxwell distribution and relaxation of the translational temperature and the temperature  related to rotations and vibrations to a common value due to equipartition of the energy in equilibrium. The model \cite{Pirner5} describes physcial situations in which the speed of relaxation of the translational and the rotational/vibrational velocity is fast compared to the ralaxation towards Maxwell distributions, whereas the model in \cite{Pirner9} covers both regimes, fast and slow relaxation. For details, see \cite{Pirner9}.  The models presented here are an extended version of \cite{Bernard} and attempt to model the two different relaxation procedures in a polyatomic gas in a different more intuitive way as it is done in \cite{Perthame}, since the two relaxation procedures are described separately. In addition, the model in \cite{Pirner9} covers fast and slow relaxation of the temperatures whereas the model in \cite{Perthame} only covers the slow relaxation of the temperatures. But since the models in \cite{Pirner5}, \cite{Pirner9} and \cite{Bernard} deal with a system of coupled equations, it is not obvious that we still have existence, uniqueness and positivity of solutions.  In this paper, we want to prove that if we use this different way of modelling the relaxation processes, we still have existence, uniqueness and positivity of solutions.

 Our aim is to prove existence, uniqueness and positivity of mild solutions of the models presented in \cite{Pirner5}, \cite{Pirner9} and \cite{Bernard}.
This work is motivated by \cite{PerthamePulvirenti}  where the global existence of mild solutions of the BGK equation for one species was established,  \cite{Yun} where global existence of mild solutions of the ES-BGK for one species is shown and \cite{Pirner3} where global existence of mild solutions of BGK models for gas mixtures is shown. There is also an existence result concerning the Boltzmann equation for mixtures in  \cite{Yun2007}. The existence and uniqueness of the model \cite{Perthame} is proven in \cite{Yunexpoly}.

The outline of the paper is as follows: in section \ref{sec1} we will present the models developed in \cite{Pirner5} and \cite{Pirner9}. In section \ref{sec3} we want to derive the macroscopic equations of this model in equilibrium in order to prove that these models produce a more general equation of state than the ideal gas law. In section \ref{sec4}, we will prove existence, uniqueness and positivity of solutions of the two models.

\section{The BGK models for a gas mixture of polyatomic molecules}
\label{sec1}
We want  to repeat the BGK models for two species of polyatomic molecules presented in \cite{Pirner5} and \cite{Pirner9} for the convenience of the reader. For more details and motivation for the choice of the model see \cite{Pirner5} and \cite{Pirner9}. The two models are very similar. They differ only in one equation, so we introduce the two models at the same time and indicate the difference with the label $(a)$ and $(b)$.
For simplicity in the following we consider a mixture composed of two different species. Let $x\in \mathbb{R}^d$ and $v\in \mathbb{R}^d, d \in \mathbb{N}$ be the phase space variables  and $t\geq 0$ the time. Let $M$ be the total number of different rotational and vibrational degrees of freedom and $l_k$ the number of internal degrees of freedom of species $k$, $k=1,2$. Note that the sum $l_1+l_2$ is not necessarily equal to $M$, because $M$ counts only the different degrees of freedom in the internal energy, $l_1+l_2$ counts all degrees of freedom in the internal energy. For example, consider two species consisting of diatomic molecules which have two rotational  degrees of freedom. In addition, the second species has one vibrational degree of freedom. Then we have $M=3, l_1=2, l_2=3$. Further, $\eta \in \mathbb{R}^{M}$  is the variable for the internal energy degrees of freedom, $\eta_{l_k} \in \mathbb{R}^{M}$ coincides with $\eta$ in the components corresponding to the internal degrees of freedom of species $k$ and is zero in the other components.  \\ Since we want to describe two different species, our kinetic model has two distribution functions $f_1(x,v,\eta_{l_1},t)> 0$ and $f_2(x,v,\eta_{l_2},t) > 0$. 
 Furthermore, for any $f_1,f_2: \Lambda_{poly} \times \mathbb{R}^d \times \mathbb{R}^M \times \mathbb{R}^+_0, \Lambda_{poly} \subset \mathbb{R}^d$ with $(1+|v|^2 + |\eta_{l_k}|^2)f_k \in L^1(dv d\eta_{l_k})$, $f_1,f_2 \geq 0,$  we relate the distribution functions to  macroscopic quantities by mean-values of $f_k$, $k=1,2$ as follows
\begin{align}
\int \int f_k(v, \eta_{l_k}) \begin{pmatrix}
1 \\ v \\ \eta_{l_k} \\ m_k |v-u_k|^2 \\ m_k |\eta_{l_k} - \bar{\eta}_{l_k} |^2 \\ m_k (v-u_k(x,t)) \otimes (v-u_k(x,t))
\end{pmatrix} 
dv d\eta_{l_k}=: \begin{pmatrix}
n_k \\ n_k u_k \\ n_k \bar{\eta}_{l_k} \\ d n_k T_k^{t} \\ l_k n_k T_k^{r} \\ \mathbb{P}_k
\end{pmatrix} , 
\label{moments}
\end{align} 
for $k=1,2$, where $n_k$ is the number density, $u_k$ the mean velocity, $\bar{\eta}_{l_k}$ the mean variable related to the internal energy, $T_k^{t}$ the mean temperature of the translation, $T_k^{r}$ the mean temperature of the internal energy degrees of freedom for example rotation or vibration and $\mathbb{P}_k$ the pressure tensor of species $k$, $k=1,2$. Note that in this paper we shall write $T_k^{t}$ and $T_k^{r}$ instead of $k_B T_k^{t}$ and $k_B T_k^{r}$, where $k_B$ is Boltzmann's constant.
In the following, we always keep the term $\bar{\eta}_{l_k}$ in order to cover the most general case, but in \cite{Bernard} and \cite{Pirner5}, they require $\bar{\eta}_{l_k}=0$, which means that the energy in rotations clockwise is the same as in rotations counter clockwise. Similar for vibrations. In addition, in the next section we will show that if one requires $\bar{\eta}_{l_k} = \omega_k$ with a fixed $\omega_k \in \mathbb{R}^M$ such that $|\omega_k|^2 = 2 \frac{p_{\infty}}{m_k n_k}$, leads to a more general equation of state in equilibrium given by $p_k= n_k T_k + const.$ 

We consider the model presented in \cite{Pirner5} given by
\begin{align} \begin{split} \label{BGK}
\partial_t f_1 + v \cdot \nabla_x  f_1   &= \nu_{11} n_1 (M_1 - f_1) + \nu_{12} n_2 (M_{12}- f_1),
\\ 
\partial_t f_2 + v \cdot \nabla_x  f_2 &=\nu_{22} n_2 (M_2 - f_2) + \nu_{21} n_1 (M_{21}- f_2), \\
f_1(t=0) &= f_1^0, \\
f_2(t=0) &= f_2^0
\end{split}
\end{align}
with the Maxwell distributions
\begin{align} 
\begin{split}
M_k(x,v,\eta_{l_k},t) &= \frac{n_k}{\sqrt{2 \pi \frac{\Lambda_k}{m_k}}^d } \frac{1}{\sqrt{2 \pi \frac{\Theta_k}{m_k}}^{l_k}} \exp({- \frac{|v-u_k|^2}{2 \frac{\Lambda_k}{m_k}}}- \frac{|\eta_{l_k}- \bar{\eta}_{l_k}|^2}{2 \frac{\Theta_k}{m_k}}), 
\\
M_{kj}(x,v,\eta_{l_k},t) &= \frac{n_{kj}}{\sqrt{2 \pi \frac{\Lambda_{kj}}{m_k}}^d } \frac{1}{\sqrt{2 \pi \frac{\Theta_{kj}}{m_k}}^{l_k}} \exp({- \frac{|v-u_{kj}|^2}{2 \frac{\Lambda_{kj}}{m_k}}}- \frac{|\eta_{l_k}- \bar{\eta}_{l_{kj}}|^2}{2 \frac{\Theta_{kj}}{m_k}}), 
\end{split}
\label{BGKmix}
\end{align}
for $ j,k =1,2, j \neq k$, 
where $\nu_{11} n_1$ and $\nu_{22} n_2$ are the collision frequencies of the particles of each species with itself, while $\nu_{12} n_2$ and $\nu_{21} n_1$ are related to interspecies collisions. 
To be flexible in choosing the relationship between the collision frequencies, we now assume the relationship
\begin{equation} 
\nu_{12}=\varepsilon \nu_{21}, \quad 0 < \frac{l_1}{l_1+l_2}\varepsilon \leq 1.
\label{coll}
\end{equation}
The restriction $\frac{l_1}{l_1+l_2} \varepsilon \leq 1$ is without loss of generality.
If $\frac{l_1}{l_1+l_2}\varepsilon >1$, exchange the notation $1$ and $2$ and choose $\frac{1}{\varepsilon}.$ In addition, we assume that all collision frequencies are positive. For the existence and uniqueness proof we assume the following restrictions on our collision frequencies
\begin{align}
\nu_{jk}(x,t) n_k(x,t) = \widetilde{\nu}_{jk} \frac{n_k(x,t)}{n_1(x,t) + n_2(x,t)}, ~ j,k =1,2
\label{asscoll}
\end{align}
with constants $\widetilde{\nu}_{11}, \widetilde{\nu}_{12}, \widetilde{\nu}_{21}, \widetilde{\nu}_{22}$.
We couple these kinetic equations with an algebraic equation for conservation of internal energy 
\begin{align}
\frac{d}{2} n_k \Lambda_k = \frac{d}{2} n_k T_k^{t} +\frac{l_k}{2} n_k T_k^{r} - \frac{l_k}{2} n_k \Theta_k, \quad k=1,2. \label{internal}
\end{align}  
and a relaxation equation ensuring that the two temperatures $\Lambda_k$ and $\Theta_k$ relax to the same value in equilibrium
 \begin{subequations}
 \begin{equation}
 \begin{split}
 \partial_t M_k + v \cdot \nabla_x M_k = \frac{\nu_{kk} n_k}{Z_r^k} \frac{d+l_k}{d} (\widetilde{M}_k - M_k)&+ \nu_{kk} n_k (M_k -f_k) \\&+ \nu_{kj} n_j (M_{kj} - f_k) , \\
 \Theta_k(0)= \Theta_k^0
 \label{kin_Temp} 
 \end{split}
 \end{equation}
 \begin{equation}
 \begin{split}
 \partial_t M_k + v \cdot \nabla_x M_k = \frac{\nu_{kk} n_k}{Z_r^k} \frac{d+l_k}{d} (\widetilde{M}_k - M_k)&+ \nu_{kj} n_j (\widetilde{M}_{kj} - M_k) , \\
 \Theta_k(0)= \Theta_k^0 
 \label{kin_Temp3}
 \end{split}
 \end{equation}
\end{subequations}
for $j,k=1,2, j \neq k$, where $Z_r^k$ are given parameters corresponding to the different rates of decays of translational and rotational/vibrational degrees of freedom. Here, we have a difference in the model presented in \cite{Pirner5} and the model in \cite{Pirner9}. The notation $(a)$ corresponds to the model in \cite{Pirner5} and the notation $(b)$ to the model in \cite{Pirner9}. 
In both cases, $M_k$ is given by
\begin{align} 
M_k(x,v,\eta_{l_k},t) = \frac{n_k}{\sqrt{2 \pi \frac{\Lambda_k}{m_k}}^d } \frac{1}{\sqrt{2 \pi \frac{\Theta_k}{m_k}}^{l_k}} \exp({- \frac{|v-u_k|^2}{2 \frac{\Lambda_k}{m_k}}}- \frac{|\eta_{l_k}- \bar{\eta}_{l_k}|^2}{2 \frac{\Theta_k}{m_k}}), \quad k=1,2,
\label{Maxwellian} 
\end{align}
and $\widetilde{M}_k$ and $\widetilde{M}_{kj}$ are given by 
\begin{align}
\widetilde{M}_k= \frac{n_k}{\sqrt{2 \pi \frac{T_k}{m_k}}^{d+l_k}} \exp \left(- \frac{m_k |v-u_k|^2}{2 T_k}- \frac{m_k|\eta_{l_k}- \bar{\eta}_{l_k}|^2}{2 T_k} \right), \quad k=1,2.
\label{Max_equ}
\end{align}
\begin{align}
\widetilde{M}_{kj}= \frac{n_k}{\sqrt{2 \pi \frac{T_{kj}}{m_k}}^{d+l_k}} \exp \left(- \frac{m_k |v-u_{kj}|^2}{2 T_{kj}}- \frac{m_k|\eta_{l_k}- \bar{\eta}_{kj,l_k}|^2}{2 T_{kj}} \right), \quad k=1,2.
\label{Max_equ3} \tag{9b}
\end{align}
where $T_k$ and $T_{kj}$ are given by 
\begin{align}
T_k:= \frac{d \Lambda_k + l_k \Theta_k}{d+l_k}= \frac{d T^{t}_k + l_k T^{r}_k}{d+l_k},
\label{equ_temp}
\end{align}
\begin{align}
T_{kj}:= \frac{d \Lambda_{kj} + l_k \Theta_{kj}}{d+l_k}. \tag{10b}
\label{equ_temp3}
\end{align}
The second equality in \eqref{equ_temp} follows from \eqref{internal}. The equation \eqref{kin_Temp} (or \eqref{kin_Temp3}) is used to involve the temperature $\Theta_k$. If we multiply \eqref{kin_Temp} (or \eqref{kin_Temp3}) by $|\eta_{l_k}|^2$, integrate with respect to $v$ and $\eta_{l_k}$ and use \eqref{equ_temp} (and also \eqref{equ_temp3} in the second model), we obtain  
\begin{subequations}
\begin{equation}
\begin{split}
\partial_t(n_k \Theta_k) +   \nabla_x\cdot (n_k \Theta_k u_k) = \frac{\nu_{kk} n_k}{Z_r^k} n_k (\Lambda_k - \Theta_k)&+ \nu_{kk} n_k n_k (\Theta_k - T_k^{r}) \\ &+ \nu_{kj} n_j n_k(\Theta_{kj} - T_k^{r}).
\end{split}
\label{relax}
\end{equation} 
\begin{equation}
\begin{split}
\partial_t(n_k \Theta_k) +   \nabla_x\cdot (n_k \Theta_k u_k) = \frac{\nu_{kk} n_k}{Z_r^k} n_k (\Lambda_k - \Theta_k)+ \nu_{kj} n_j n_k(T_{kj} - \Theta_k).
\end{split}
\label{relax3}
\end{equation}
\end{subequations}
for $k=1,2$. The initial data of $\Lambda_k$ and $\Lambda_k$ itself is determined using \eqref{internal}.

The Maxwell distributions $M_1$ and $M_2$ in \eqref{BGKmix} have the same densities, mean velocities and internal energies as $f_1$ and $f_2$, respectively. With this choice, we guarantee the conservation of the number of particles, momentum and internal energy in interactions of one species with itself (see section 3.2 in \cite{Pirner5}). The remaining parameters $n_{12}, n_{21}, u_{12}, u_{21}, \Lambda_{12}, \Lambda_{21}, \Theta_{12}$ and $\Theta_{21}$ will be determined determined using conservation of the number of particles, of total momentum and total energy, together with some symmetry considerations. 

If we assume that \begin{align} n_{12}=n_1 \quad \text{and} \quad n_{21}=n_2,  
\label{density2} 
\end{align}
we have conservation of the number of particles, see theorem 2.1 in \cite{Pirner}.
If we further assume 
 \begin{align}
u_{12}&= \delta u_1 + (1- \delta) u_2, \quad \delta \in \mathbb{R},
\label{convexvel2}
\end{align} then we have conservation of total momentum
provided that
\begin{align}
u_{21}=u_2 - \frac{m_1}{m_2} \varepsilon (1- \delta ) (u_2 - u_1), 
\label{veloc2}
\end{align}
see theorem 2.2 in \cite{Pirner}.

 In \cite{Pirner5}  it is assumed that $\bar{\eta}_{l_1} = \bar{\eta}_{l_2}=0$. In order to give a proof for the most general case, we do not make this assumption. If we do not make this assumption, we also need corresponding definitions for $\bar{\eta}_{l_1,12}$ and $\bar{\eta}_{l_2,21}$. This is done in the next definition.
\begin{definition}
\label{defeta}
We consider
\begin{align*}
\bar{\eta}_{12} &= \beta \bar{\eta}_{l_1} + (1- \beta) \bar{\eta}_{l_2}, \quad \beta \in \mathbb{R},
\end{align*}
and define $\bar{\eta}_{l_1,12}$ as the vector which is equal to $\bar{\eta}_{12}$ in the components where $\eta_{l_1}$ coincides with $\eta$ and zero otherwise. In addition, consider
\begin{align*}
\bar{\eta}_{21}=\bar{\eta}_{l_2} - \frac{m_1}{m_2} \varepsilon (1- \beta ) (\bar{\eta}_{l_2} - \bar{\eta}_{l_1})
\end{align*}
and define $\bar{\eta}_{l_2,21}$ as the vector which is equal to $\bar{\eta}_{21}$ in the components where $\eta_{l_2}$ coincides with $\eta$ and zero otherwise.
\end{definition}
Similar as in the case of the mean velocities one can prove that this definition leads to conservation of momentum.

If we further assume that $\Lambda_{12}$ and $\Theta_{12}$ are of the following form
\begin{align}
\begin{split}
\Lambda_{12} &=  \alpha \Lambda_1 + ( 1 - \alpha) \Lambda_2 + \gamma |u_1 - u_2 | ^2,  \quad 0 \leq \alpha \leq 1, \gamma \geq 0 ,\\
\Theta_{12} &= \frac{l_1 \Theta_1 + l_2 \Theta_2}{l_1 +l_2} + \tilde{\gamma} | \bar{\eta}_{l_1} - \bar{\eta}_{l_2}|^2, \quad \quad \quad \quad \quad \quad \quad ~ \tilde{\gamma} \geq 0,
\label{contemp2}
\end{split}
\end{align}
then we have conservation of total energy and a uniform choice of the temperatures 
provided that
\begin{align}
\begin{split}
\Lambda_{21} &=\left[ \frac{1}{d} \varepsilon m_1 (1- \delta) \left( \frac{m_1}{m_2} \varepsilon ( \delta - 1) + \delta +1 \right) - \varepsilon \gamma \right] |u_1 - u_2|^2 \\&+ \varepsilon ( 1 - \alpha ) \Lambda_1 + ( 1- \varepsilon ( 1 - \alpha)) \Lambda_2, \\
\Theta_{21}&= \varepsilon \frac{l_1}{l_1 +l_2} \Theta_1 + \left( 1- \varepsilon \frac{l_1}{l_1+l_2} \right) \Theta_2 - \frac{l_1}{l_2} \varepsilon \tilde{\gamma} |\bar{\eta}_{l_1} - \bar{\eta}_{l_2}|^2 \\&- \varepsilon \frac{m_1}{l_2} (|\bar{\eta}_{l_1,12}|^2 - |\bar{\eta}_{l_1}|^2) - \frac{m_2}{l_2} (|\bar{\eta}_{l_2,21}|^2 -|\bar{\eta}_{l_2}|^2)
\label{temp2}
\end{split}
\end{align}
see theorem 3.2 and remark 3.2 in \cite{Pirner5}.
In order to ensure the positivity of all temperatures, we need to restrict $\delta$, $\beta$, $\gamma$ and $\tilde{\gamma}$ to 
 \begin{align}
 \begin{split}
0 \leq \gamma  \leq \frac{m_1}{d} (1-\delta) \left[(1 + \frac{m_1}{m_2} \varepsilon ) \delta + 1 - \frac{m_1}{m_2} \varepsilon \right], \\
0 \leq \tilde{\gamma}  \leq \frac{m_1}{l_1} (1-\beta) \left[(1 + \frac{m_1}{m_2} \varepsilon ) \beta + 1 - \frac{m_1}{m_2} \varepsilon \right], 
\end{split}
 \label{gamma2}
 \end{align}
and
\begin{align}
\begin{split}
 \frac{ \frac{m_1}{m_2}\varepsilon - 1}{1+\frac{m_1}{m_2}\varepsilon} \leq  \delta \leq 1, \quad
  \frac{ \frac{m_1}{m_2}\varepsilon - 1}{1+\frac{m_1}{m_2}\varepsilon} \leq  \beta \leq 1,
  \end{split}
\label{gammapos2}
\end{align}
see theorem 2.5 in \cite{Pirner} for $N=3$ in the mono atomic case. 

For the convenience of the reader, we want to summarize our models in order to clarify which equation and definitions belong to model 1 and which to model 2. In both models we use equation \eqref{BGK} with definitions \eqref{BGKmix}, \eqref{coll}, \eqref{asscoll}, \eqref{density2}, \eqref{convexvel2}, \eqref{veloc2}, definition \ref{defeta}, \eqref{contemp2} and \eqref{temp2} for the time evolution of $f_1$ and $f_2$. In order to evolve $\Theta_1$ and $\Theta_2$, we couple equation \eqref{BGK} in model 1 with equation \eqref{kin_Temp} with definitions \eqref{Maxwellian}, \eqref{Max_equ} and \eqref{equ_temp}, whereas model 2 uses equation \eqref{kin_Temp3} with definitions \eqref{Maxwellian}, \eqref{Max_equ}, \eqref{Max_equ3}, \eqref{equ_temp} and \eqref{equ_temp3}. Both models are then coupled with equation \eqref{internal} to determine $\Lambda_1$ and $\Lambda_2$.
\section{Equation of state in the macroscopic equations}
\label{sec3}
In this section we want to illustrate the effect of the additional variable $\eta$ on the equation of state in the macroscopic equations. We want to illustrate this in the case of one species.
We consider a distribution function $f(x,v,\eta,t)> 0$ introduced in the previous section for one species. 
We relate the distribution function to  macroscopic quantities by mean-values of $f$ as follows
\begin{align}
\int \int f(v, \eta) \begin{pmatrix}
1 \\ v \\ \eta \\ m~ |v-u|^2 \\ m~ |\eta - \bar{\eta}_l |^2 \\ 
\end{pmatrix} 
dv d\eta =: \begin{pmatrix}
n \\ n~ u \\ n~ \bar{\eta}_l \\ 3~ n~ T^{t} \\ l~ n~ T^{r} \\ 
\end{pmatrix}.
\label{moments10}
\end{align} 
Now, assume that we assume that $\bar{\eta}_l$ is fixed and equal to a vector  $w$  in $\mathbb{R}^l$ such that $|w|^2= 2 \frac{p_{\infty}}{m n}$ for a given constant $p_{\infty}$ in the Maxwell distribution in \eqref{BGKmix}.   Since $|w|^2$ represents the kinetic energy in the rotation and vibration, $p_{\infty}$ may be related to the moment of inertia in the case of rotations or the Hook'sches law in the case of vibrations. In this case, we will obtain an equation of state given by 
$$p = n T + \text{const.}$$  This is shown in the following. The additional constant takes into account an attractive force between the particles which is neglected in the case of an ideal gas.

\begin{theorem}[Macroscopic equations]
Assume $f$ decays fast enough to zero in the $v$ and $\eta$ variables and is a solution to \eqref{BGK}. If in addition $f$ is in equilibrium meaning it is a Maxwell distribution and the temperatures $T^t$ and $T^r$ are equal to $T$,  it satisfies the following local macroscopic conservation laws.
\vspace{-0.5cm}
\begin{multline*}
\\
\partial_t n + \nabla_x(n u)=0
\\
 \partial_t(m n u)+\nabla_x \left(n T \right) + \nabla_x \cdot (m n u \otimes u  ) = 
0,
\\
\partial_t \left(\frac{m}{2} n |u|^2 + \frac{3+l}{2} n T \right) + \nabla_x \cdot \left(\left(\frac{5+l}{2} n T + p_{\infty} \right) u \right)+ \nabla_x \cdot \left(\frac{m}{2} n |u|^2 u \right)  = 0,
\\
\end{multline*}
\end{theorem}
\vspace{-0.7cm}
\begin{proof}
If we integrate equation \eqref{BGK} with respect to $v$ and $\eta$ and use $f=M$, we get:
\begin{align*}
\int \int \partial_t M dv d\eta + \int \int v \cdot \nabla_x   M dv d\eta = 0.
\end{align*}
If we formally exchange integration and derivatives, we obtain
\begin{align*}
\partial_t \int \int  M dv d\eta + \nabla_x \cdot \int \int  v M dv d\eta = 0.
\end{align*}
This is equivalent to
\begin{align*}
\partial_t n + \nabla_x \cdot (n u)=0,
\end{align*}
since we have
{\small
\begin{align*}
\int \int M dv d\eta &= \int  \frac{n}{\sqrt{2 \pi \frac{T}{m}}^3 }  \exp \left(- \frac{|v-u|^2}{2 \frac{T}{m}} \right) dv \int \frac{1}{\sqrt{2 \pi \frac{T}{m}}^{l}} \exp \left(- \frac{|\eta- w|^2}{2 \frac{T}{m}} \right) d\eta \\&=n
\end{align*}}
and
{\small
\begin{align*}
\int \int M v dv d\eta &= \int v \frac{n}{\sqrt{2 \pi \frac{T}{m}}^3 }  \exp \left(- \frac{|v-u|^2}{2 \frac{T}{m}} \right) dv \int \frac{1}{\sqrt{2 \pi \frac{T}{m}}^{l}} \exp \left(- \frac{|\eta- w|^2}{2 \frac{T}{m}} \right) d\eta \\&=n u.
\end{align*}}
Multiplying the equation \eqref{BGK} by $m v$, integrating it with respect to  $v$ and $\eta$ and using that $f$ is equal to a Maxwell distribution with temperatures equal to $T$, leads to 
\begin{align*}
m \int \int v \partial_t M dv d\eta + m \int \int v ~ v \cdot \nabla_x M dv d\eta  = 0.
\end{align*}
We formally exchange derivative and integration and obtain 
\begin{align*}
m ~\partial_t(n u) + \nabla_x \cdot \int \int m v \otimes v M dv d\eta =0 .
\end{align*}
We can compute
{\footnotesize
\begin{align*}
\int \int v \otimes v M dv d\eta &= \int v \otimes v \frac{n}{\sqrt{2 \pi \frac{T}{m}}^3 }  \exp \left(- \frac{|v-u|^2}{2 \frac{T}{m}} \right) dv \int \frac{1}{\sqrt{2 \pi \frac{T}{m}}^{l}} \exp \left(- \frac{|\eta- w|^2}{2 \frac{T}{m}} \right) d\eta \\&=n u \otimes u + n \frac{T}{m},
\end{align*}}
so the second term turns into
\begin{align*}
\nabla_x (nT) + \nabla_x \cdot (m n u \otimes u  ) .
\end{align*}
So all in all, we get
\begin{align*}
 \partial_t(m n u)+\nabla_x (nT) + \nabla_x \cdot (m n u \otimes u  )  = 
0 .
\end{align*}
Multiplying the equation \eqref{BGK} by $\frac{m}{2}(|v|^2 + |\eta|^2)$, integrating it with respect to $v$ and $\eta$ and using that $f$ is a Maxwell distribution with temperatures equal to $T$, leads to
\begin{align*}
\frac{m}{2} \int \int (|v|^2 + |\eta|^2) \partial_t M dv d\eta + \frac{m}{2} \int \int(|v|^2 + |\eta|^2) v \cdot  \nabla_x M dv d\eta
= 0.
\end{align*}
We formally exchange derivative and integration and obtain
\begin{align*}
 \partial_t \left(\frac{m}{2} n |u|^2 + \frac{3+l}{2} n T \right) + \nabla_x \cdot \int \int m v (|v|^2 +|\eta|^2) M dv d\eta=0,
\end{align*}
since we have
\begin{align*}
m \int (|v|^2 + |\eta|^2) M dv d\eta &= \frac{m n |u|^2}{2} + \frac{3}{2} n T + \frac{mn |w|^2}{2} + \frac{l}{2} n T \\ &= \frac{m n |u|^2}{2} + \frac{3+l}{2} n T + p_{\infty}
\end{align*}
where $p_{\infty}$ is a constant, so its time derivative vanishes.
Last, we compute
\begin{align*}
m &\int v (|v|^2 + |\eta|^2) M dv d\eta = \int |v|^2 v M dv d\eta + \int |\eta|^2 v M dv d\eta \\&= \left(\frac{mn}{2} |u|^2 + \frac{3}{2} n T \right) u \\&+ \int v \frac{n}{\sqrt{2 \pi \frac{T}{m}}^3 }  \exp \left(- \frac{|v-u|^2}{2 \frac{T}{m}} \right) dv \int |\eta|^2 \frac{1}{\sqrt{2 \pi \frac{T}{m}}^{l}} \exp \left(- \frac{|\eta- w|^2}{2 \frac{T}{m}} \right) d\eta \\&= \left(\frac{mn}{2} |u|^2 + \frac{3}{2} n T \right) u + \left(\frac{l}{2} n T + p_{\infty} \right) u
\end{align*}
and obtain
\begin{align*}
\partial_t \left(\frac{m}{2} n |u|^2 + \frac{3+l}{2} n T  \right) + \nabla_x \cdot \left[\left( \frac{mn}{2} |u|^2 + \frac{5+d}{2} n T + p_{\infty} \right) u \right] = 0. 
\end{align*}
\end{proof}
\section{Existence, Uniqueness and Positivity of solutions }
\label{sec4}
\subsection{Existence and Uniqueness of mild solutions}

 In the following, we want to study mild solutions of \eqref{BGK} coupled with \eqref{internal} and \eqref{kin_Temp}, and mild solutions of \eqref{BGK} coupled wirh \eqref{internal} and \eqref{kin_Temp3}. For a simpler handling later in the existence and uniqueness proof for model 1, we first arrange our system \eqref{BGK} and \eqref{kin_Temp} to the following equivalent system. We define $z_k= Z^k_r \frac{d}{d+l_k}$ and
 $$ g_k=  M_k - f_k,$$ and then we consider the following mild formulation of model 1 given by
\begin{definition}\label{milddef}
We call $(f_1, f_2, M_1, M_2)$ with $(1+|v|^2+ |\eta_{l_k}|^2)f_k, M_k \in L^1(dv d\eta_{l_k}),$ $ f_1,f_2, M_1, M_2 \geq 0$ a mild solution to \eqref{BGK} coupled with \eqref{kin_Temp} and \eqref{internal} under the conditions of the collision frequencies \eqref{asscoll} if and only if $f_1,f_2, M_1, M_2$ satisfy
{\footnotesize
\begin{align*}
\begin{split}
&f_k(x,v,\eta_{l_k},t)= e^{-\alpha_k(x,v,t)} f_k^0(x-tv,v, \eta_{l_k}) \\ &+ e^{-\alpha_k(x,v,t)} \int_0^t [ \widetilde{\nu}_{kk} \frac{n_k(x+(s-t)v,s)}{n_k(x+(s-t)v,s)+ n_j(x+(s-t)v,s)} M_k(x+(s-t)v,v,\eta_{l_k},s) \\ &+ \widetilde{\nu}_{kj} \frac{n_j(x+(s-t)v,s)}{n_k(x+(s-t)v,s)+ n_j(x+(s-t)v,s)} M_{kj}(x+(s-t)v,v,\eta_{l_k},s)] e^{\alpha_k(x+(s-t)v,v,s)} ds,
\end{split}
\end{align*}
\begin{align*}
\begin{split}
&g_k(x,v,\eta_{l_k},t)=  g_k^0(x-tv,v, \eta_{l_k}) \\ &+  \int_0^t [ \frac{\widetilde{\nu}_{kk}}{z_r} \frac{n_k(x+(s-t)v,s)}{n_k(x+(s-t)v,s)+ n_j(x+(s-t)v,s)} \\&(\widetilde{M}_k(x+(s-t)v,v,\eta_{l_k},s)- M_k(x+(s-t)v,v,\eta_{l_k},s)) ]
  ds,
\end{split}
\end{align*}}
where $\alpha_k$ is given by
{\small
\begin{align*}
\begin{split}
\alpha_k(x,v,t) = \int_0^t [\widetilde{\nu}_{kk} \frac{n_k(x+(s-t)v,s)}{n_k(x+(s-t)v,s) + n_j(x+(s-t)v,s) }\\ +\widetilde{\nu}_{kj} \frac{n_j(x+(s-t)v,s)}{n_k(x+(s-t)v,s) + n_j(x+(s-t)v,s) } ] ds,
\end{split}
\end{align*}}
and $$ M_k = g_k+f_k.$$
for $k,j =1,2, ~k\neq j$.
\end{definition}
For model 2, we consider the following mild formulation
\begin{definition} \label{milddef2} 
We call $(f_1, f_2, M_1, M_2)$ with $(1+|v|^2+ |\eta_{l_k}|^2)f_k, M_k \in L^1(dv d\eta_{l_k}),$ $ f_1,f_2, M_1, M_2 \geq 0$ a mild solution to \eqref{BGK} coupled with \eqref{kin_Temp3} and \eqref{internal} under the conditions of the collision frequencies \eqref{asscoll} if and only if $f_1,f_2, M_1, M_2$ satisfy
{\footnotesize
\begin{align*}
\begin{split}
&f_k(x,v,\eta_{l_k},t)= e^{-\alpha_k(x,v,t)} f_k^0(x-tv,v, \eta_{l_k}) \\ &+ e^{-\alpha_k(x,v,t)} \int_0^t [ \widetilde{\nu}_{kk} \frac{n_k(x+(s-t)v,s)}{n_k(x+(s-t)v,s)+ n_j(x+(s-t)v,s)} M_k(x+(s-t)v,v,\eta_{l_k},s) \\ &+ \widetilde{\nu}_{kj} \frac{n_j(x+(s-t)v,s)}{n_k(x+(s-t)v,s)+ n_j(x+(s-t)v,s)} M_{kj}(x+(s-t)v,v,\eta_{l_k},s)] e^{\alpha_k(x+(s-t)v,v,s)} ds,
\end{split}
\end{align*}
\begin{align*}
\begin{split}
&M_k(x,v,\eta_{l_k},t)=  M_k^0(x-tv,v, \eta_{l_k}) +  \int_0^t [ \frac{\widetilde{\nu}_{kk}}{z_r} \frac{n_k(x+(s-t)v,s)}{n_k(x+(s-t)v,s)+ n_j(x+(s-t)v,s)} \\&(\widetilde{M}_k(x+(s-t)v,v,\eta_{l_k},s)- M_k(x+(s-t)v,v,\eta_{l_k},s)) ]\\&+ \widetilde{\nu}_{kj} \frac{n_j(x+(s-t)v,s)}{n_k(x+(s-t)v,s)+ n_j(x+(s-t)v,s)}\\& ( \widetilde{M}_{kj}(x+(s-t)v,v,\eta_{l_k},s) - M_k(x+(s-t)v,v,\eta_{l_k},s))
 ds,
\end{split}
\end{align*}}
where $\alpha_k$ is given by
{\small
\begin{align*}
\begin{split}
\alpha_k(x,v,t) = \int_0^t [\widetilde{\nu}_{kk} \frac{n_k(x+(s-t)v,s)}{n_k(x+(s-t)v,s) + n_j(x+(s-t)v,s) }\\ +\widetilde{\nu}_{kj} \frac{n_j(x+(s-t)v,s)}{n_k(x+(s-t)v,s) + n_j(x+(s-t)v,s) } ] ds,
\end{split}
\end{align*}}
for $k,j =1,2, ~k\neq j$.
\end{definition}
First, we present some estimates on macroscopic quantities which we need later for the existence and uniqueness proof. 
\begin{theorem}
Define the variable $\xi_{l_k} = (v, \eta_{l_k}), \bar{\xi}_{l_k}=(u,\bar{\eta}_{l_k})$. For any  $(f_1, f_2,$ $ M_1, M_2)$ with $(1+|\xi_{l_k}|^2)f_k, (1+|\xi_{l_k}|^2)M_k \in L^1(d\xi_{l_k}),$ $ f_1,f_2, M_1, M_2 \geq 0$, we define the moments and macroscopic parameters as in \eqref{moments}, \eqref{convexvel2}, \eqref{veloc2}, definition \ref{defeta}, \eqref{contemp2} and \eqref{temp2} and set
\begin{align}
N_q(f_k)(\xi_{l_k})= \sup_{\xi_{l_k}} |\xi_{l_k}|^q f_k(\xi_{l_k}), \quad q\geq0, k=1,2.
\label{Nq}
\end{align}
Then the following estimates hold 
\begin{enumerate}
\item[(i.1)] $\frac{n_k}{T_k^{(d+l_k)/2}} \leq C N_0(f_k),~\frac{n_k}{T_k^{(d+l_k)/2}} \leq C N_0(M_k), \frac{n_k}{\Lambda_k^{d/2}} \leq C N_0(M_k), ~ \frac{n_k}{\Theta_k^{l_k/2}} \leq C N_0(M_k),\quad  k=1,2,$

\item[(i.2)] $ \frac{n_1}{\Lambda_{12}^{d/2}} \leq C N_0(M_1),~ \frac{n_1}{\Theta_{12}^{l_k/2}} \leq C N_0(M_1),$ 
\item[(i.3)] $ \frac{n_1}{\Lambda_{21}^{d/2}} \leq C N_0(M_2),~ \frac{n_1}{\Theta_{21}^{l_k/2}} \leq C N_0(M_2).$
\item[(i.2b)/(i.3b)] $\frac{n_k}{T_{kj}^{(d+l_k)/2}} \leq C N_0(M_k),$ for $k,j=1,2, ~ k\neq j$
\end{enumerate}
\label{theoest1}
\end{theorem}
\begin{proof}
The proof of $(i.1)$ is analougous to the proof of the inequality $(2.2)$ in \cite{PerthamePulvirenti}. For the first inequality in $(i.1)$ replace $|v-u_k|^2$ by $|\xi_{l_k}-\bar{\xi}_{l_k}|^2$ and repeat all the steps done there. For the second inequality replace $|v-u_k|^2$ by $|\xi_{l_k}-\bar{\xi}_{l_k}|^2$ and use $M_k$ instead of $f_k$. For the third inequality replace $|v-u_k|^2$ by $|\xi_{l_k}-\bar{\xi}_{l_k}|^2$ use $M_k$ instead of $f_k$, and for the last inequality replace $|v-u_k|^2$ by $|\eta_{l_k} - \bar{\eta}_{l_k}|^2$ and also $f_k$ by $M_k$.  


We deduce the estimates $(i.2)$ and $(i.3)$ from $(i.1)$. This is done in the same way as in the mono atomic case done in theorem 3.1.1 in \cite{Pirner3}.

The proof of $(i.2b)/(i.3b)$ is similar to the proof of $(ii.2)$ and $(ii.3)$ and is therefore omitted here.

\end{proof}
\begin{theorem}
For any pair of functions $(f_1, f_2, M_1, M_2)$ with $(1+|\xi_{l_k}|^2)f_k, (1+|\xi_{l_k}|^2)M_k \in L^1(d\xi_{l_k}),$ $ f_1,f_2,$ $ M_1, M_2 \geq 0$, we define the moments as in \eqref{moments}, \eqref{convexvel2}, \eqref{veloc2}, \eqref{contemp2} and \eqref{temp2}, then we have
\begin{enumerate}
\item[(ii.1)] $n_k (T_k +|u_k|^2 +|\bar{\eta}_{l_k}|^2)^{\frac{q-d-l_k}{2}} \leq C_q N_q(f_k)$ for $q>d+l_k+2$, $k=1,2$,\\
$n_k (T_k +|u_k|^2 +|\bar{\eta}_{l_k}|^2)^{\frac{q-d-l_k}{2}} \leq C_q N_q(M_k)$ for $q>d+l_k+2$, $k=1,2$,\\
$n_k (\Lambda_k +|u_k|^2 |^2)^{\frac{q-d}{2}} \leq C_q N_q(M_k)$ for $q>d+2$, $k=1,2$, \\
$n_k (\Theta_k +|\bar{\eta}_{l_k}|^2 )^{\frac{q-l_k}{2}} \leq C_q N_q(M_k)$ for $q>l_k+2$, $k=1,2$,
\item[(ii.2)] $n_1 (\Lambda_{12} +|u_{12}|^2)^{\frac{q-d}{2}} \leq C_q (N_q(M_1) + \frac{n_1}{n_2} N_q(M_2))$ for $q>d+2$,\\
$n_1 (\Theta_{12} +|\bar{\eta}_{l_1,12}|^2)^{\frac{q-d}{2}} \leq C_q (N_q(M_1) + \frac{n_1}{n_2} N_q(M_2))$ for $q>l_k+2$
\item[(ii.3)] $n_2 (\Lambda_{21} +|u_{21}|^2 )^{\frac{q-d}{2}} \leq C_q (\frac{n_2}{n_1} N_q(M_1) + N_q(M_2))$ for $q>d+2$, \\
$n_2 (\Theta_{21} +|\bar{\eta}_{l_2,21}|^2 )^{\frac{q-l_k}{2}} \leq C_q (\frac{n_2}{n_1} N_q(M_1) + N_q(M_2))$ for $q>l_k+2$.
\item[(ii.2b)/(ii.3b)] $n_1 (T_{12} +|u_{12}|^2 +|\bar{\eta}_{l_1,12}|^2)^{\frac{q-d}{2}} \leq C_q (N_q(M_1) + \frac{n_1}{n_2} N_q(M_2))$ \\
$n_2 (T_{21} +|u_{21}|^2 +|\bar{\eta}_{l_2,21}|^2 )^{\frac{q-l_k}{2}} \leq C_q (\frac{n_2}{n_1} N_q(M_1) + N_q(M_2))$ for $q>d+l_k+2$.
\end{enumerate}
\label{theoest2}
\end{theorem}
\begin{proof}
The proof of $(ii.1)$ is analougous to the proof of the inequality $(2.3)$ in \cite{PerthamePulvirenti}. For the first inequality in $(ii.1)$ replace $v$ by $\xi_{l_k}$ and repeat all the steps done there. For the second inequality replace $v$ by $\xi_{l_k}$ and use $M_k$ instead of $f_k$.  For the third inequality use $M_k$ instead of $f_k$, and for the last inequality replace $v$ by $\eta_{l_k}$ and also $f_k$ by $M_k$. 


The proof of $(ii.2)$ and $(ii.3)$ is analougous to the proof of the inequalities $(ii.2)$ and $(ii.3)$ in theorem 3.1.2 in \cite{Pirner3}. For the first estimates in $(ii.2)$ and $(ii.3)$, replace $f_k$ by $M_k$, for the second estimate replace in addition $v$ by $\eta_{l_k}$.

If ywe insert the definition of $T_{kj}$ into $(ii.2b)/(ii.3b)$, we can estimate the left-hand side of $(ii.2b)/(ii.3b)$ with lemma 3.1.3 in \cite{Pirner3} in terms of the right-hand sides of $(ii.2)$ and $(ii.3)$ and apply the inequalities $(ii.2)$ and $(ii.3)$.

\end{proof}
\begin{theorem}
For any pair of functions $(f_1, f_2, M_1, M_2)$ with $(1+|\xi|^2)f_k,(1+|\xi|^2)M_k \in L^1(d\xi),$ $ f_1,f_2,$ $ M_1, M_2 \geq 0$, we define the moments as in \eqref{moments}, \eqref{convexvel2}, \eqref{veloc2}, \eqref{contemp2} and \eqref{temp2}, then we have
\begin{enumerate}
\item[(iii.1)] $\frac{n_k |\bar{\xi}_{l_k}|^{d+l_k+q}}{[(T_k+|\bar{\xi}_k|^2) T_k]^{d+l_k/2}} \leq C_q N_q(f_k)$ for any $q>1, k=1,2$,\\ $\frac{n_k |\bar{\xi}_{l_k}|^{d+l_k+q}}{[(T_k+|\bar{\xi}_k|^2) T_k]^{d+l_k/2}} \leq C_q N_q(M_k)$ for any $q>1, k=1,2$,\\$\frac{n_k |u_k|^{d+q}}{[(\Lambda_k+|u_k|^2) \Lambda_k]^{d/2}} \leq C_q N_q(M_k)$ for any $q>1, k=1,2$, \\ $\frac{n_k |\bar{\eta}_{l_k}|^{l_k+q}}{[(\Lambda_k+|\bar{\eta}_{l_k}|^2) \Lambda_k]^{l_k/2}} \leq C_q N_q(M_k)$ for any $q>1, k=1,2$, 
\item[(iii.2)] $\frac{n_1 |u_{12}|^q}{\Lambda_{12}^{d/2}} \leq n_1 C ( \frac{|u_1|^q}{( \Lambda_1)^{d/2}} + \frac{|u_2|^q}{(\Lambda_2)^{d/2}})$ for any $q>1$,\\ $\frac{n_1 |\bar{\eta}_{l_1,12}|^q}{\Theta_{12}^{l_1/2}} \leq n_1 C ( \frac{|\bar{\eta}_{l_1}|^q}{( \Theta_1)^{l_1/2}} + \frac{|\bar{\eta}_{l_2}|^q}{(\Theta_2)^{l_1/2}})$ for any $q>1$,
\item[(iii.3)]$\frac{n_2 |u_{21}|^q}{\Lambda_{21}^{l_2/2}} \leq n_2 C ( \frac{|u_1|^q}{( \Lambda_1)^{l_1/2}} + \frac{|u_2|^q}{(\Lambda_2)^{l_2/2}})$ for any $q>1$, \\ $\frac{n_2 |\bar{\eta}_{l_2,21}|^q}{\Theta_{21}^{l_2/2}} \leq n_2 C ( \frac{|\bar{\eta}_{l_1}|^q}{( \Theta_1)^{l_2/2}} + \frac{|\bar{\eta}_{l_2}|^q}{(\Theta_2)^{l_2/2}})$ for any $q>1$.
\item[(iii.2b)/(iii.3b)]  $\frac{n_1 |\bar{\xi}_{l_1,12}|^q}{T_{12}^{(d+l_1)/2}} \leq n_1 C ( \frac{|\bar{\xi}_1|^q}{( T_1)^{(d+l_2)/2}} + \frac{|\bar{\xi}_2|^q}{(T_2)^{(d+l_2)/2}})$ for any $q>1$,
\end{enumerate}
\label{theoest3}
\end{theorem}
\begin{proof}
The proof of $(iii.1)$ is analougous to the proof of the inequality $(2.3)$ in \cite{PerthamePulvirenti}. For the first inequality in $(iii.1)$ replace $v$ by $\xi_{l_k}$ and repeat all the steps done there. For the second inequality replace $v$ by $\xi_{l_k}$ and use $M_k$ instead of $f_k$. For the third inequality use $M_k$ instead of $f_k$, and for the last inequality replace $v$ by $\eta_{l_k}$ and also $f_k$ by $M_k$. 

The proof of $(iii.2)$, $(iii.3)$ and $(iii.2b)/(iii.3b)$ is analougous to the proof of the inequalities $(iii.2)$ and $(iii.3)$ in theorem 3.1.4 in \cite{Pirner3}. 
%
\end{proof}
\begin{cons}
For any functions $(f_1, f_2, M_1, M_2)$ with $(1+|\xi|^2)f_k \in L^1(d\xi),$ $ f_1,f_2, M_1, M_2 \geq 0$, we define the moments as in \eqref{moments}, \eqref{convexvel2}, \eqref{veloc2}, \eqref{contemp2} and \eqref{temp2}, then we have
\begin{enumerate}
\item[(iv.1)] $\sup_{\xi_{l_k}} |\xi_{l_k}|^q \widetilde{M}_k[f_k] \leq C_q N_q(f_k)$ for $q>d+l_k+2$ or $q=0$, \\
$\sup_{\xi_{l_k}} |\xi_{l_k}|^q \widetilde{M}_k[f_k] \leq C_q N_q(M_k)$ for $q>d+l_k+2$ or $q=0$,
\item[(iv.2)] $\sup_{\xi_{l_k}} |\xi_{l_k}|^q M_{12}[M_1, M_2] \leq C_q ( N_q(M_1) + \frac{n_1}{n_2} N_q(M_2))$ for $q>d+l_k+2$ or $q=0$,
\item[(iv.3)] $\sup_{\xi_{l_k}} |\xi_{l_k}|^q M_{21}[M_1, M_2] \leq C_q ( \frac{n_2}{n_1} N_q(M_1) +  N_q(M_2))$ for $q>d+l_k+2$ or $q=0$.
\item[(iv.2b)/(iv.3b)] $\sup_{\xi_{l_k}} |\xi_{l_k}|^q \widetilde{M}_{12}[M_1, M_2] \leq C_q ( N_q(M_1) + \frac{n_1}{n_2} N_q(M_2))$ for $q>d+l_k+2$ or $q=0$, \\ $\sup_{\xi_{l_k}} |\xi_{l_k}|^q \widetilde{M}_{21}[M_1, M_2] \leq C_q ( \frac{n_2}{n_1} N_q(M_1) +  N_q(M_2))$ for $q>d+l_k+2$ or $q=0$.
\end{enumerate}
\label{conest}
Note that here and in the following we write $M_k[f_k], M_{12}[f_1,f_2]$ and $M_{21}[f_1,f_2]$ instead of $M_k, M_{12}$ and $M_{21}$ in order to emphasize the dependence of the Maxwell distributions on the distribution functions $f_1$ and $f_2$ via the macroscopic quantities as densities, velocities and temperatures.
\end{cons}
\begin{proof}
The proof of $(iv.1)$ is analougous as the proof of the inequality $(2.3)$ in \cite{PerthamePulvirenti} exchanging $v$ by $\xi$ and $M_k$ by $\widetilde{M}_k$ using the estimates $(i.1)$, $(ii.1)$ and $(iii.1)$ for $\widetilde{M}_k$.


The proof of $(iv.2)$, $(iv.3)$ and $(iv.2b)/(iv.3b)$ is very similar to the proof $(iv.2)$ and $(iv.3)$ in consequences 3.1.5 in \cite{Pirner3} exchanging $f_k$ by $M_k$ and in addition using lemma 3.1.3 in \cite{Pirner3} in the beginning in order to estimate $|\xi_{l_k}|^q =(|u_k|^2 + |\eta_{l_k}|^2)^{\frac{q}{2}}$ from above by $C(|u_k|^q + |\eta_{l_k}|^q)$.
\end{proof}
Now, we want to show existence and uniqueness of non-negative solutions in a certain function space using the previous estimates. We prove only the existence and uniqueness of model 1, since the proof of model 2 is analogous to the proof of model 2. This is, because the term $\nu_{kj} n_j (\widetilde{M}_{kj} - M_k)$ can be handled in the same way as the term $\frac{\nu_{kk} n_k}{Z_r^k} \frac{d+l_k}{d} (\widetilde{M}_k - M_k)$.
For the existence and uniqueness proof, we make the following assumptions:
\begin{ass}
\begin{enumerate}
\item We assume periodic boundary conditions. Equivalently we can construct solutions satisfying 
{\footnotesize
$$f_k(t,x_1,..., x_d, (\xi_{l_k})_1,...,(\xi_{l_k})_{d+l_k})= f_k(t,x_1,...,x_{i-1},x_i + a_i,x_{i+1},...x_d,(\xi_{l_k})_1,...(\xi_{l_k})_{d+l_k}),$$
$$M_k(t,x_1,..., x_d, (\xi_{l_k})_1,...,(\xi_{l_k})_{d+l_k})= M_k(t,x_1,...,x_{i-1},x_i + a_i,x_{i+1},...x_d,(\xi_{l_k})_1,...(\xi_{l_k})_{d+l_k}),$$}
for all $i=1,...,d+l_k$ and a suitable $\lbrace a_i\rbrace \in \mathbb{R}^d$ with positive components, for $k=1,2$.
\item We require that the initial values $f_k^0, M_k^0 k=1,2$ satisfy assumption $1$.
\item We are on the bounded domain in space $\Lambda_{poly}=\lbrace x \in \mathbb{R}^d | x_i \in (0,a_i)\rbrace$.
\item Suppose that $f_k^0$ satisfies $f_k^0 \geq 0$, $(1+|\xi_{l_k}|^2) f_k^0 \in L^1(\Lambda_{poly} \times \mathbb{R}^N)$ with \\$\int f_k^0 dx dv =1, k=1,2$.
\item Suppose $N_q(f_k^0):= \sup_{\xi_{l_k}} f_k^0(x,\xi_{l_k})(1+|\xi_{l_k}|^q) = \frac{1}{2} A_0 < \infty$ and $N_q(g_k^0):= \sup_{\xi_{l_k}} g_k^0(x,\xi_{l_k})(1+|\xi_{l_k}|^q) = \frac{1}{2} A_0 < \infty$ for some $q>d+l_k+2$.
\item Suppose $\gamma_k(x,t):= \int f_k^0(x-vt,v, \eta_{l_k}) dvd\eta_{l_k} \geq C_0 >0$ for all $t\in\mathbb{R}.$
\item Assume that the collision frequencies are written as in \eqref{asscoll} and are positive.
\item Assume that the initial data $\Theta_k^0$, $\Lambda_k^0$ satisfy condition \eqref{internal} and are integrable with respect to $x\in \Lambda_{poly}$
\item Assume that the relaxation parameter in front of $(\widetilde{M}-M)$ in \eqref{BGK} is in $L^{\infty}(dx)$ and non-negative.
\end{enumerate}
\label{ass2}
\end{ass}
With this assumptions we can show the following theorem.
\begin{theorem}
Under the assumptions \ref{ass2} and the definitions \eqref{moments}, \eqref{density2}, \eqref{convexvel2}, \eqref{veloc2}, definition \ref{defeta}, \eqref{contemp2} and \eqref{temp2}, there exists a unique non-negative mild solution $(f_1,f_2, M_1, M_2)\in C(\mathbb{R}^+ ; L^1((1+ |v|^2) dv dx))$ of the initial value problem \eqref{BGK} coupled with \eqref{kin_Temp} and \eqref{internal}. Moreover, for all $t>0$ the following bounds hold:
\begin{align*}
|u_k(t)|, |u_{12}(t)|, |u_{21}(t)|, |\eta_{l_k}|, |\bar{\eta}_{l_k}|, |\bar{\eta}_{l_k}|, T_k(t), T_{12}(t), T_{21}(t), N_q(f_k)(t) \leq A(t) &< \infty, \\
n_k(t) \geq C_0 e^{-(\widetilde{\nu}_{kk} + \widetilde{\nu}_{kj})t} &>0, \\
T_k(t), \Lambda_k(t), \Theta_k(t), \Lambda_{12}(t), \Theta_{12}(t), \Lambda_{21}(t), \Theta_{21}(t) \geq B(t)&>0,
\end{align*} 
for $k=1,2$ and some constants $A(t),B(t)$ given by
$$ A(t) = C e^{Ct}, ~ B(t) = C e^{-Ct}, ~ C>0$$
\label{ex2}
\end{theorem}
\vspace{-1cm}
\begin{proof}
The idea of the proof is to find a Cauchy sequence of functions in a certain space which converges towards a solution to \eqref{BGK} coupled with \eqref{kin_Temp} and \eqref{internal}. The sequence will be constructed in a way such that each member of the sequence satisfies an inhomogeneous transport equation. In this case we know results of existence and uniqueness. In order to show that this sequence is a Cauchy sequence we need to show that the Maxwell distributions on the right-hand side of \eqref{BGK} and \eqref{kin_Temp} are Lipschitz continuous with respect to $f_1,f_2$ and $M_1,M_2$, respectively.\\ \\ The proof is structured as follows: First, we prove some estimates on the macroscopic quantities \eqref{moments}, \eqref{density2}, \eqref{convexvel2}, definition \ref{defeta}, \eqref{veloc2}, \eqref{contemp2} and \eqref{temp2}. From this we can deduce Lipschitz continuity of the Maxwell distributions $\widetilde{M}_k, M_{12}, M_{21}$ with respect to $M_1$ and $M_2$ which finally leads to the convergence of this Cauchy sequence to a solution to \eqref{BGK} coupled with \eqref{kin_Temp}. \\ \\ \textbf{ Step 1}: Gronwall estimate on $N_q(f_k(t))$ given by \eqref{Nq} \\ \\
If $f_1$ is part of a mild solution according to definition \ref{milddef}, we have
{\footnotesize
\begin{align*}
&N_q(f_1)= \sup_{\xi_{l_1}} |\xi_{l_1}|^q f_1 \leq  e^{-\alpha_1(x,v,t)} \sup_{\xi_{l_1}} |\xi_{l_1}|^q f_1^0(x-tv,v, \eta_{l_k}) \\ &+ \sup_{\xi_{l_1}} |\xi_{l_1}|^q [ e^{-\alpha_1(x,v,t)} \int_0^t [ \widetilde{\nu}_{11} \frac{n_1(x+(s-t)v,s)}{n_1(x+(s-t)v,s)+ n_2(x+(s-t)v,s)} M_1(x+(s-t)v,v,\eta_{l_k},s) \\ &+ \widetilde{\nu}_{12} \frac{n_2(x+(s-t)v,s)}{n_1(x+(s-t)v,s)+ n_2(x+(s-t)v,s)} M_{12}(x+(s-t)v,v,\eta_{l_k}, s)] e^{\alpha_1(x+(s-t)v,v,s)} ds ].
\end{align*}}
Since $\alpha_1$ is non-negative, we can estimate $e^{- \alpha_1(x,v,t)}$ in front of the initial data from above by $1$. Since we assumed that the collision frequencies have the shape given in \eqref{asscoll}, we can estimate the integrand in the exponential function \\$e^{- \alpha_1(x,v,t)} e^{\alpha_1(x+ (s-t)v,v,s)}$ by a constant and obtain
{\footnotesize
\begin{align*}
&N_q(f_1)= \sup_{\xi_{l_1}} |\xi_{l_1}|^q f_1 \leq  \sup_{\xi_{l_1}} |\xi_{l_1}|^q f_1^0(x-tv,v,\eta_{l_k}) \\ &+ \sup_{\xi_{l_1}} |\xi_{l_1}|^q [ \int_0^t e^{-C(t-s)} [ C \frac{n_1(x+(s-t)v,s)}{n_1(x+(s-t)v,s)+ n_2(x+(s-t)v,s)} M_1(x+(s-t)v,v,\eta_{l_k},s) \\ &+ C \frac{n_2(x+(s-t)v,s)}{n_1(x+(s-t)v,s)+ n_2(x+(s-t)v,s)} M_{12}(x+(s-t)v,v,\eta_{l_k},s)]  ds ].
\end{align*}}
Using assumption 5 (in the assumption \ref{ass2}) and the fact that  we can estimate $e^{-C(t-s)}$ from above by $1$ since $s$ is between $0$ and $t$, we get
\begin{align*}
&N_q(f_1)= \sup_{\xi_{l_1}} |\xi_{l_1}|^q f_1 \leq  \frac{1}{2} A_0 +   \int_0^t C \sup_x [  \frac{n_1(x,s)}{n_1(x,s)+ n_2(x,s)} \sup_{\xi_{l_1}} |\xi_{l_1}|^q M_1(x,v,\eta_{l_k},s) \\ &+  \frac{n_2(x,s)}{n_1(x,s)+ n_2(x,s)} \sup_{\xi_{l_1}} |\xi_{l_1}|^q M_{12}(x,v,\eta_{l_k},s)]  ds .
\end{align*}
With  $(iv.2),$ we obtain
{\footnotesize
\begin{align*}
N_q(f_1)&= \sup_{x,\xi_{l_1}} |\xi_{l_1}|^q f_1 \\ &\leq  \frac{1}{2} A_0 +   \int_0^t C_q \sup_x [\frac{n_1(x,s)+n_2(x,t)}{n_1(x,s)+ n_1(x,s)}  N_q(M_1)(s) +  \frac{n_1(x,s)}{n_1(x,s)+ n_2(x,s)} N_q(M_2(s))]  ds \\ &\leq \frac{1}{2} A_0 +   \int_0^t C_q [  \sup_x N_q(M_1)(s) +  \sup_x N_q(M_2)(s)]  ds.
\end{align*}}
Similarly, we can estimate $N_q(f_2)$ by
\begin{align*}
N_q(f_2)&= \sup_{\xi_{l_2}} |\xi_{l_2}|^q f_2 \leq  \frac{1}{2} A_0 +   \int_0^t C_q [ \sup_x N_q(M_1)(s) +  \sup_x N_q(M_2)(s)]  ds.
\end{align*}
We add both inequalities and obtain
\begin{align}
 N_q(f_1)+N_q(f_2) \leq A_0 + \int_0^t C_q [\sup_x N_q(M_1)(s)+ \sup_x N_q(M_2)(s)] ds.
 \label{Gron1}
 \end{align}

Now, if $g_1$ is part of a mild solution according to definition \ref{milddef}, we have
{\footnotesize
\begin{align*}
&N_q(g_1)= \sup_{\xi_{l_1}} |\xi_{l_1}|^q g_1 \leq   \sup_{\xi_{l_1}} |\xi_{l_1}|^q g_1^0(x-tv,v, \eta_{l_1}) \\ &+ \sup_{\xi_{l_1}} |\xi_{l_1}|^q [  \int_0^t [ \frac{\widetilde{\nu}}{z_r^1}_{11} \frac{n_1(x+(s-t)v,s)}{n_1(x+(s-t)v,s)+ n_2(x+(s-t)v,s)}\\& (\widetilde{M}_1(x+(s-t)v,v,\eta_{l_1},s)-M_1(x+(s-t)v,v,\eta_{l_1},s)) ]  ds ].
\end{align*}}
We estimate $-M_1$ by $M_1$ and  get
{\footnotesize
\begin{align*}
&N_q(g_1)= \sup_{\xi_{l_1}} |\xi_{l_1}|^q g_1 \leq  \sup_{\xi_{l_1}} |\xi_{l_1}|^q g_1^0(x-tv,v,\eta_{l_1}) \\ &+ \sup_{\xi_{l_1}} |\xi_{l_1}|^q [ \int_0^t e^{-C(t-s)} [ C \frac{n_1(x+(s-t)v,s)}{n_1(x+(s-t)v,s)+ n_2(x+(s-t)v,s)} \\&(\widetilde{M}_1(x+(s-t)v,v,\eta_{l_k},s)+M_1(x+(s-t)v,v,\eta_{l_k},s)) ]  ds ].
\end{align*}}
Using assumption 5 (in the assumption \ref{ass2}), we get
{\small
\begin{align*}
&N_q(g_1)= \sup_{\xi_{l_1}} |\xi_{l_1}|^q g_1 \leq  \frac{1}{2} A_0 \\&+   \int_0^t C \sup_x [  \frac{n_1(x,s)}{n_1(x,s)+ n_2(x,s)} \left(\sup_{\xi_{l_1}} |\xi_{l_1}|^q \widetilde{M}_1(x,v,\eta_{l_k},s)+\sup_{\xi_{l_1}} |\xi_{l_1}|^q M_1(x,v,\eta_{l_k},s)\right) ]  ds .
\end{align*}}
With  $(iv.1),$ we obtain
{\footnotesize
\begin{align*}
N_q(g_1)= \sup_{x,\xi_{l_k}} |\xi_{l_k}|^q g_1 \leq  \frac{1}{2} A_0 +   \int_0^t C_q \left( \sup_x [  N_q(f_1)(s)+\sup_x [  N_q(M_1)(s)]\right)   ds.
\end{align*}}
Similarly, we can estimate $N_q(g_2)$ by
\begin{align*}
N_q(g_2)&= \sup_{\xi_{l_2}} |\xi_{l_2}|^q g_2 \leq  \frac{1}{2} A_0 +   \int_0^t C_q \left( \sup_x N_q(f_2)(s)+\sup_x N_q(M_2)(s)\right)  ds.
\end{align*}
We add both inequalities and obtain
\begin{align}
\begin{split}
 &N_q(g_1)+N_q(g_2) \leq A_0 \\&+ \int_0^t C_q [\sup_x N_q(f_1)(s)+ \sup_x N_q(f_2)(s)+\sup_x N_q(M_1)(s)+ \sup_x N_q(M_2)(s)] ds.
 \end{split}
 \label{Gron2}
 \end{align}
 Now, we add the inequalities \eqref{Gron1} and \eqref{Gron2} and obtain
 \begin{align*}
 N_q(f_1)+N_q(f_2) + N_q(g_1) + N_q(g_2) \leq 2 A_0 \\+ \int_0^t C_q [ \sup_x N_q(f_1)(s) + \sup_x N_q(f_2)(s) + 2\sup_x N_q(M_1)(s) +2 \sup_x N_q(M_2)(s) ] ds
 \end{align*}
 According to the definition of $g_k$, we have $M_k= g_k+f_k$ and therefore we get
  \begin{align*}
 N_q(f_1)+N_q(f_2) + N_q(M_1) + N_q(M_2) \leq 2 A_0\\+ \int_0^t 4 C_q [ \sup_x N_q(f_1)(s) + \sup_x N_q(f_2)(s) + \sup_x N_q(M_1)(s) + \sup_x N_q(M_2)(s) ] ds
 \end{align*}
With Gronwall's lemma, we obtain
\begin{align}
N_q(f_1)+N_q(f_2) + N_q(M_1) + N_q(M_2) \leq 2 A_0  e^{4 C_q t} 
\label{G}
\end{align}
for $q>d+l_1+l_2+2$ or $q=0.$
\\ \\ \textbf{ Step 2}: Estimate on the densities \\ \\
The proof is analougous to the proof in the mono atomic case given in \cite{Pirner3}.
\\  \\
\textbf{ Step 3}: Estimate on the temperatures \\ \\
The estimate on the temperatures $T_k,\Lambda_k,\Theta_k$ from below can be proven analougously as in step 3 in \cite{Pirner3} now using the extended estimates in $(i.1),$ \eqref{G} and the estimate on the density from the previous step stated in this theorem. 

The proof of the estimates on $\Lambda_{12}, \Lambda_{21}, \Theta_{12}, \Theta_{21}$ are also analogeous as in step 3 in \cite{Pirner3} in the mono atomic case using $(i.2)$,$(i.3)$, \eqref{G} and the estimate on the density from step 2. 
\\ \\
\textbf{ Step 4}: Estimates on the velocities and temperatures from above\\ \\
The estimates on $\Lambda_k, \Theta_k, T_k$, $\Lambda_{12}, \Theta_{12}$, $\Lambda_{21}, \Theta_{21}$, $|u_k|, |\bar{\eta}_{l_k}|$, $|u_{12}|, |u_{21}|, |\bar{\eta}_{l_1,12}|,$ $ |\bar{\eta}_{l_2,21}|$ can be estimated in an analougeous way as in step for in \cite{Pirner3} now using the estimates $(ii.1), (ii.2)$ and $(ii.3)$, \eqref{G} and the estimate on the densities from step 2.
\\ \\ \textbf{ Step 5}: Lipschitz continuity \\ \\
The next step of the proof is to show Lipschitz continuity of the operators $(M_k,M_j)$ $ \mapsto \widetilde{\nu}_{kk} \frac{n_k}{n_k+n_j} M_k + \widetilde{\nu}_{kj} \frac{n_j}{n_k+n_j} M_{kj}[M_k,M_j]$, $M_k \mapsto \frac{\widetilde{\nu}_{kk}}{z_r^k} \frac{n_k}{n_k+n_j} (\widetilde{M}_{k}[M_k]-M_k)$, when \\$(f_1,f_2,\Lambda_1, \Lambda_2,\Theta_1,\Theta_2)$ are restricted to
{\footnotesize
\begin{align}
\begin{split}
\Omega=\lbrace f_1,f_2\in L^1(\Lambda_{poly} \times \mathbb{R}^N; (1+|v|^2) dv dx), \Lambda_1,\Lambda_2,\Theta_1, \Theta_2 \in L^1(dx) |\\ f_k \geq 0, N_q(f_k)<A, \min (n_k,T_k,\Lambda_k, \Theta_k)>C,k=1,2 \rbrace.
\end{split}
\label{Omega}
\end{align}
}
The proof for $(M_k,M_j) \mapsto \widetilde{\nu}_{kk} \frac{n_k}{n_k+n_j} M_k + \widetilde{\nu}_{kj} \frac{n_j}{n_k+n_j} M_{kj}[M_k,M_j]$ is analogous to the mono atomic case done in \cite{Pirner3}  replacing $(f_k,f_j)$ by $(M_k,M_j)$.

It remains to prove the Lipschitz continuity of $\widetilde{M}_k$ with respect to $M_k$. The proof is analogous  to the proof for $f_k \mapsto M_k[f_k]$ in the mono atomic case given in \cite{PerthamePulvirenti} using the whole internal energy $|v|^2+ |\eta_{l_k}|^2$ instead of only $|v|^2$.
\\ \\ \textbf{ Step 6}: Existence and Uniqueness of non-negative solutions in $\bar{\Omega}$ (see definition of $\Omega$ in \eqref{Omega})\\ \\
Now, introduce the sequence  $\lbrace(f_1^n,f_2^n, \Lambda_1^n, \Lambda_2^n, \Theta_1^n, \Theta_2^n)\rbrace$ of mild solutions to
\begin{align*} 
\begin{split} 
\partial_t f_k^n + v\cdot\nabla_x  f_k^n   &= \widetilde{\nu}_{kk} \frac{n_k^{n-1}}{n_k^{n-1}+n_j^{n-1}} (M_k^{n-1|n-2} - f_k^n) \\ &+ \widetilde{\nu}_{kj} \frac{n_j^{n-1}}{n_k^{n-1}+n_j^{n-1}} (M_{kj}^{n-1|n-2}- f_k^n),
\\ 
\partial_t g_k^{n-1|n-2} + v\cdot\nabla_x   g_k^{n-1|n-2}   &= \frac{\widetilde{\nu}_{kk}}{z_r^k} \frac{n_k^{n-2}}{n_k^{n-2}+n_j^{n-2}} (\widetilde{M}_k^{n-2} - g_k^{n-2|n-3}- f_k^{n-2}) \\ \frac{d}{2} \Lambda_k^{n-1} + \frac{l_k}{2} \Theta_k^{n-2} &= \frac{d}{2} T_k^{trans,n-2} + \frac{l_k}{2} T_k^{rot,n-2} \\ f_1^0&= f_1(t=0), \\ f_2^0&=f_2(t=0) \\ \Theta_k^0&= \Theta_k(0), \quad \quad \quad  n \geq 3
\end{split}
\end{align*}
for $k,j=1,2, k\neq j$. The meaning of the notation $n-1|n-2$ is the following. In distribution functions with this index we take the value of $\Theta_k^{n-2}$ but all the other macroscopic quantities have the index $n-1$. Since the zeroth functions are known as the initial values, these are inhomogeneous transport equations for fixed $n\in \mathbb{N}$. For an inhomogeneous transport equation we know the existence of a unique mild solution in the periodic setting

Now, we show that $\lbrace(f_1^n,f_2^n, \Lambda_1^n, \Lambda_2^n, \Theta_1^n, \Theta_2^n)\rbrace$ is a Cauchy sequence in $\Omega$. Then, since $\bar{\Omega}$ is complete, we can conclude convergence in $\bar{\Omega}$. First, we show that $\lbrace(f_1^n,f_2^n, \Lambda_1^n, \Lambda_2^n, \Theta_1^n, \Theta_2^n)\rbrace$ is in $\Omega$.
\begin{itemize}
\item $f_1^n,f_2^n$ are in $L^1((1+|v|^2)dv dx)$ since $f_1^0,f_2^0$ are in $L^1((1+|v|^2)dv dx)$.
\item $\Theta_1^n, \Theta_2^n$ are in $L^1(dx)$ since $\Theta_1^0, \Theta_2^0$ are in $L^1(dx)$.
\item $f_1^n,f_2^n\geq 0$ since $f_1^0,f_2^0\geq 0$.
\item $N_q(f_k^n) < A$, $\min(n_k^n,T_k^n, \Lambda_k^n, \Theta_k^n)>C$, since all estimates in step $1,2$ and $4$ are independent of $n$.
\end{itemize}
Now, $\lbrace(f^n_1,f^n_2)\rbrace$ is a Cauchy sequence in $\Omega$ since we have
{\scriptsize
\begin{align*}
&||f_1^n-f_1^{n-1}||_{L^1((1+|\xi|^2)d\xi_{l_1} dx)}\\ &\leq \int_{\Lambda_{poly}} \int_{\mathbb{R}^d} e^{-\alpha_1^{n-1}(x,v,t)} \int_0^t e^{\alpha_1^{n-1}(x+(s-t)v,v,s)} \big|\widetilde{\nu}_{11}^{n-1} \frac{n_1^{n-1}(x+(s-t)v,s)}{n_1^{n-1}(x+(s-t)v,s)+ n_2^{n-1}(x+(s-t)v,s)}\\ &M_1^{n-1|n-2}(x+(s-t)v,v,\eta_{l_1},s) - \widetilde{\nu}_{11}^{n-2} \frac{n_1^{n-2}(x+(s-t)v,s)}{n_1^{n-2}(x+(s-t)v,s)+ n_2^{n-2}(x+(s-t)v,s)} \\&M_1^{n-2|n-3}(x+(s-t)v,v,\eta_{l_1},s)\big|ds (1+|\xi_{l_1}|^2) dx d\xi_{l_1} \\&+ \int_{\Lambda_{poly}} \int_{\mathbb{R}^d} e^{-\alpha_1^{n-1}(x,v,t)} \int_0^t e^{\alpha_1^{n-1}(x+(s-t)v,v,\eta_{l_1},s)} \big|\widetilde{\nu}_{12}^{n-1} \frac{n_2^{n-1}(x+(s-t)v,s)}{n_1^{n-1}(x+(s-t)v,s)+ n_2^{n-1}(x+(s-t)v,s)}\\ &M_{12}^{n-1|n-2}(x+(s-t)v,v,\eta_{l_1}, s) - \widetilde{\nu}_{12}^{n-2} \frac{n_2^{n-2}(x+(s-t)v,s)}{n_1^{n-2}(x+(s-t)v,s)+ n_2^{n-2}(x+(s-t)v,s)} \\ &M_{12}^{n-2|n-3}(x+(s-t)v,v,\eta_{l_1},s)\big|ds (1+|\xi_{l_1}|^2) dx d\xi_{l_1}.
\end{align*}}
Now we use the Lipschitz continuity of the Maxwell distributions
{\scriptsize
\begin{align*}
&||f_1^n-f_1^{n-1}||_{L^1((1+|\xi_{l_1}|^2)d\xi_{l_1} dx)}\\ &\leq C \int_{\Lambda_{poly}} \int_{\mathbb{R}^d} e^{-\alpha_1^{n-1}(x,v,t)} \int_0^t e^{\alpha_1^{n-1}(x+(s-t)v,v,s)} | M_1^{n-1|n-2}(x+(s-t)v,v, \eta_{l_1},s) \\&- M_1^{n-2|n-3}(x+(s-t)v,v,\eta_{l_1},s)|ds (1+|\xi_{l_1}|^2) dx d\xi_{l_1} \\&+ \int_{\Lambda_{poly}} \int_{\mathbb{R}^d} e^{-\alpha_1^{n-1}(x,v,t)} \int_0^t e^{\alpha_1^{n-1}(x+(s-t)v,v,s)} [ | M_{1}^{n-1|n-2}(x+(s-t)v,v,\eta_{l_1},s) \\&- M_{1}^{n-2|n-3}(x+(s-t)v,v,\eta_{l_1},s)|\\&+| M_{2}^{n-1|n-2}(x+(s-t)v,v,\eta_{l_2},s) - M_{2}^{n-2|n-3}(x+(s-t)v,v,\eta_{l_2},s)|] ds (1+|\xi_{l_2}|^2) dx d\xi_{l_2}
\\ &\leq   e^{-Ct} \int_0^t e^{Cs} [ || M_{1}^{n-1|n-2}(s) - M_{1}^{n-2|n-3}(s)||_{L^1((1+|\xi_{l_1}|^2)d\xi_{l_1} dx}\\&+|| M_{2}^{n-1|n-2}(s) - M_{2}^{n-2|n-3}(s)||_{L^1((1+|\xi_{l_2}|^2)d\xi_{l_2} dx}] ds.
\end{align*}}
Similarly, we get for species $2$
{\footnotesize
\begin{align*}
&||f_2^n-f_2^{n-1}||_{L^1((1+|\xi_{l_2}|^2)d\xi_{l_2} dx)} \\ &\leq   e^{-Ct} \int_0^t e^{Cs} [ || M_{1}^{n-1|n-2}(s) - M_{1}^{n-2|n-3}(s)||_{L^1((1+|\xi_{l_2}|^2)d\xi_{l_2} dx)}\\&+|| M_{2}^{n-1|n-2}(s) - M_{2}^{n-2|n-3}(s)||_{L^1((1+|\xi_{l_2}|^2)d\xi_{l_2} dx)}] ds.
\end{align*}}
Now, we use the definition $M_k = g_k + f_k$ and replace $M_k^{n-1|n-2}$ and $M_k^{n-2|n-3}$ by $g_k^{n-1|n-2}+ f_k^{n-1}$ and $g_k^{n-2|n-3} + f_k^{n-2}$, respectively. By triangle inequality, we obtain
\begin{align*}
&||f_2^n-f_2^{n-1}||_{L^1((1+|\xi_{l_2}|^2)d\xi_{l_2} dx)} \\&\leq   e^{-Ct} \int_0^t e^{Cs} [ || g_{1}^{n-1|n-2}(s) - g_{1}^{n-2|n-3}(s)||_{L^1((1+|\xi_{l_2}|^2)d\xi_{l_2} dx)}\\&+|| g_{2}^{n-1|n-2}(s) - g_{2}^{n-2|n-3}(s)||_{L^1((1+|\xi_{l_2}|^2)d\xi_{l_2} dx)}\\&+ || f_{1}^{n-1}(s) - f_{1}^{n-2}(s)||_{L^1((1+|\xi_{l_2}|^2)d\xi_{l_2} dx)}\\&+|| f_{2}^{n-1}(s) - f_{2}^{n-2}(s)||_{L^1((1+|\xi_{l_2}|^2)d\xi_{l_2} dx)}] ds.
\end{align*}
Now, we insert the mild formulation for $g_k$ in order to replace the terms $$|| g_{k}^{n-1|n-2}(s) - g_{k}^{n-2|n-3}(s)||_{L^1((1+|\xi_{l_k}|^2)d\xi_{l_k} dx)}.$$ If we do this, we will obtains additional terms with $|| M_{k}^{n-2|n-3}(s) - g_{k}^{n-3|n-4}(s)||_{L^1((1+|\xi_{l_k}|^2)d\xi_{l_k} dx)}$ and a term with $||\widetilde{M}_k^{n-2}- \widetilde{M}_k^{n-3}||_{L^1((1+|\xi_{l_k}|^2)d\xi_{l_k} dx)}$. On the last term we apply the Lipschitz continuity of $\widetilde{M}_k$ with respect to $M_k$. This gives again terms with $f_k$ and $g_k$. Doing this inductively one can prove in an analougous way as in Step 6 in \cite{Pirner3}, that $\lbrace f_k^n \rbrace$ and $\lbrace g_k^n \rbrace$ are Cauchy sequences in $\bar{\Omega}$. Additionally, one can conclude existence and uniqueness in the same way as in Step 6 of \cite{Pirner3} for the mono atomic case.

\end{proof}
\subsection{Positivity of solutions}
\begin{theorem}
Let $(f_1,f_2, M_1, M_2)$ be a mild solution to \eqref{BGK} coupled with \eqref{kin_Temp} and \eqref{internal} (or \eqref{BGK} coupled with \eqref{kin_Temp3} and \eqref{internal}) under the modified assumptions for existence and uniqueness described in the previous section with positive initial data. Then the solution is positive meaning $f_1, f_2, M_1, M_2 >0$. 
\end{theorem}
\begin{proof}
The proof is exactly the same as in the case of the BGK model for mixtures, see section 4 in \cite{Pirner3}.
\end{proof}

\section*{Acknowledgments} The author thanks Christian Klingenberg for suggesting this topic and encouraging me to write it up. Thanks go to him and Gabriella Puppo for many discussions on polyatomic modelling and multi-species kinetics.


\end{document}